\pdfoutput=1
\documentclass[letterpaper]{amsart}
\usepackage{amssymb}
\usepackage{enumerate}
\usepackage{graphicx}
\usepackage[usenames,dvipsnames]{xcolor}
\usepackage[final]{hyperref}
\usepackage{microtype}
\usepackage[normalem]{ulem}

\newtheorem{theorem}{Theorem}

\theoremstyle{definition}
\newtheorem{definition}{Definition}
\newtheorem*{remark}{Remark}

\newcommand{\AND}{\operatorname{\textsc{and}}}
\newcommand{\OR}{\operatorname{\textsc{or}}}

\newcommand{\R}		{\mathbb{R}}
\newcommand{\N}		{\mathbb{N}}

\renewcommand{\Pr}[2][]	{\mathbf{P}_{#1}\!\left( #2 \right)}
\newcommand{\Ex}[1]	{\mathbf{E}\left( #1 \right)}
\newcommand{\E}		{\mathbf{E}}

\newcommand{\Var}		{\mathbf{Var}}

\newcommand{\isdef}	{\overset{\rm def}{=}}
\newcommand{\inlaw}	{\overset{\mathcal{L}}{=}}
\newcommand{\tendsinlaw} {\overset{\mathcal{L}}{\to}}

\begin{document}

\title{Recursive functions on conditional Galton--Watson trees}

\author{Nicolas Broutin}
\address{Sorbonne Universit\'{e}, Campus Pierre et Marie Curie Case courrier 158, 4, place Jussieu 75252 Paris Cedex 05 France.}
\email{nicolas.broutin@upmc.fr}

\author{Luc Devroye}
\address{School of Computer Science, McGill University, 3480 University Street, Montreal, Canada H3A 2K6}
\email{lucdevroye@gmail.com}

\author{Nicolas Fraiman}
\address{Department of Statistics and Operations Research, University of North Carolina at Chapel Hill, 337 Hanes Hall, Chapel Hill, NC 27599, USA}
\email{fraiman@email.unc.edu}

\subjclass[2010]{60J80,60J85,60G99}
\keywords{Random Galton--Watson tree, probabilistic analysis of algorithms, recursive functions, branching process.}

\thanks{Luc Devroye was sponsored by NSERC Grant A3456. Nicolas Broutin thanks the FRQNT-CRM and Simons-CRM programmes as well as the support from Grant ANR-14-CE25-0014 (ANR GRAAL)} 

\date{\today}

\begin{abstract}
A recursive function on a tree is a function in which each leaf
has a given value, and each internal node has a value equal to
a function of the number of children, the values of the children,
and possibly an explicitly specified random element $U$.
The value of the root is the key quantity of interest in general.
In this study, all node values and function values
are in a finite set $S$. 
In this note, we describe the limit behavior when the leaf values
are drawn independently from a fixed distribution on $S$, and
the tree $T_n$ is a random Galton--Watson tree of size $n$.
\end{abstract}

\maketitle

\section{The probabilistic model}

A recursive function on a tree is a function in which each leaf
has a given value, and each internal node has a value equal to
a function of the number of children, the values of the children,
and possibly an explicitly specified random element $U$.
The value of the root is the key quantity of interest in general.
In the present study, all node values and function values
are in a finite set $S$, and we describe the limit behavior when the leaf values
are drawn independently from a fixed distribution on $S$, and
the tree $T_n$ is a random Galton--Watson tree of size $n$.

A Galton--Watson (or Galton--Watson--Bienaym\'e) tree { (see \cite{AN72})}
is a rooted random ordered tree. Each node independently 
generates a random number of children drawn from a fixed offspring distribution $\xi$.
The distribution of $\xi$ defines the distribution of $T$,
a random Galton--Watson tree.
We define
$$
p_i = \Pr{ \xi = i }, i \ge 0.
$$
The results are sometimes described in terms of the generating
function $g$ of $\xi$:
$$
g(s) \isdef \Ex{ s^\xi } = \sum_{i=0}^\infty p_i s^i, \quad \text{for } 0 \le s \le 1.
$$
In what follows, we are mainly interested in critical Galton--Watson trees, i.e., those
having { $\Ex{\xi} = 1 $}, and $\Pr{ \xi = 1 } < 1$. 
In addition, we assume that the variance of $\xi$ is positive and finite.
We denote by $T_n$ a Galton--Watson tree conditional on its size $|T_n|$ being $n$.
These trees encompass many known models of random trees,
including random Catalan trees (all binary trees of size $n$ being equally
likely), random planted plane trees (all ordered trees being equally likely),
and random rooted labeled free trees or Cayley trees,
thanks to an equivalence property first established by 
{ Kennedy in \cite{Ken75}.}
Let $h = \hbox{\rm gcd} \{ i \ge 1: p_i > 0 \}$ be the span of
$\xi$.  It is easy to see that $|T_n| \mod h = 1$, so 
when we provide asymptotic results on $T_n$, it is
understood that $n \mod h=1$ as $n \to \infty$.


Nodes in a tree are denoted by $u, v$ and $w$, while
their values are denoted by $V(u)$, $V(v)$ and $V(w)$.
Without loss of generality,
we assume that our state space is
$$
S = \{ 1 , \ldots , k \}.
$$
We associate independently with each node 
a copy of a generic uniform $[0,1]$ random variable $U$.
Thus, $U(v)$ denotes the copy associated with node $v$.
We are given a possibly infinite family of functions
$$
f_0, f_1, f_2, \ldots,
$$
where $f_i$ maps $S^i \times [0,1]$ to $S$.  The first $i$
arguments refer to the values of the $i$ children of a node, while
the last argument refers to the generic random variable associated
with a node. In particular, { for} each leaf $v$, we have 
$$
V(v) \inlaw f_0 (U(v)).
$$
Thus, the leaf values are independent
and we denote the distribution of $f_0 (U)$  on $S$ { by $q$:}
$$
\Pr{ f_0 (U) = i } = q_i , \quad i \in S.
$$
If $v$ is an internal node with children $v_1,\ldots,v_\ell$, then
$$
V(v) = f_\ell (V(v_1),\ldots,V(v_\ell), U(v) ).
$$
The value of the root node is denoted by $V_n$.

For a { path}, with the root having value $V_n$ and the other nodes having values $V_{n-1}, V_{n-2}, \dots, V_1, V_0$, we have $V_0=f_0(U_0)$, $V_1=f_1(V_0,U_1)$, $V_2 = f_1(V_1, U_1)$, and so forth. This is a purely Markovian structure. The limit behavior is entirely known and well-documented in standard texts
on Markov chains { such as \cite{MT93}.} 
The decomposition of the transition matrix graph (which places a directed edge for every transition from
$i$ to $j$ in $S$ that has nonzero probability) is of prime importance. The most interesting case is
that of the existence of just one irreducible strongly connected component. In that case, $V_n$ either tends (in distribution)
to a stationary limit random variable or exhibits a periodic behavior if the period of the irreducible set
is more than one.

We exclude { paths} throughout the manuscript by requiring that $p_1 \not= 1$ (or, equivalently, $\Var(\xi) > 0$).

\medskip

\section{Recursive functions on random Galton--Watson trees}

As a warm-up, we need to study the behavior of the value
of the root of $T$, an unconditional critical Galton--Watson tree. 
This case has been treated thoroughly by Aldous and Bandyopadhyay { in \cite{AlBa2005}}; we will come back to their contribution shortly. 
Since $|T| < \infty$ with probability one,
the root's value, $W$, is a properly defined random variable.
What matters is its support set, that is, the set of all
possible values $W$ can take.  This support set includes
the support set of the leaf values. 
Note that the support set of $V_n$ is a subset of the support
set of $W$.  As we see later, it can be a proper subset.

Since there is no use for values of $S$ that are not in the
support set of $W$, without loss of generality we define 
$S$ as the support set of $W$.

We are not concerned with the precise derivation of the
law of $W$. It suffices to say that it is \emph{a solution} of
the distributional identity
$$
W \inlaw f_\xi (W_1,\ldots,W_\xi, U),
$$
where $W, W_1, W_2,\ldots$ are independent an identically distributed (i.i.d.),  and $U$, $\xi$ and
$W_1, W_2, \ldots$ are independent 
(indeed, without any additional condition, this equation { may admit more than one solution}). Worked out examples follow later.

\medskip
\begin{remark}
In their paper \cite{AlBa2005}, Aldous and Bandyopadhyay investigated this very fixed point equation, and it is in this context that the question of the representation of the solution as an \emph{unconditioned} Galton--Watson tree arose: if { one} expands the distributional fixed-point equation into a tree, the tree obtained is a Galton--Watson tree and the fixed-point can be represented by such a tree. Now, one of the main questions they address is the following: when is the value at the root measurable with respect to the sigma-algebra generated by the random variables in the { tree?} { (This sigma-algebra must include the information about the shape of the tree, as a subset of $\cup_{n\ge 0} \N^n$ for instance.)} When this is the case, the system is called endogenous. This question of endogeny is only interesting when the tree is infinite, and in the present case of a critical Galton--Watson { tree}, the answer is trivial. However, we shall see soon that some of the conditions they had for endogeny are intimately related to the condition for convergence in the context of Galton--Watson trees conditioned on being infinite. 
\end{remark}
\medskip

\section{Coalescent Markov chains}

We deal with an explicit Markov chain governed by
$$
X_{t}  = f(X_{t-1} , U_t),
$$
where the $U_t$'s are independent random elements with distribution $\mu$, and $f$
is a function that maps to the finite state space $S$.

{
{
\begin{definition}\label{def:coal}
We call this Markov chain \emph{coalescent} if the double Markov chain 
$$
(X_{t},Y_{t}) = ( f(X_{t-1}, U_t), f(Y_{t-1}, U_t))
$$
defined using the {\it same} random elements $U_t$, $t>0$, in both maps is such that: for any starting point $(X_0,Y_0)=(x,y)\in S^2$, with probability one, $X_t=Y_t$ for all $t$ large enough.
\end{definition}
}


Definition~\ref{def:coal} is a version of what Aldous and Bandyopadhyay 
call bivariate uniqueness; see Section 2.4 in \cite{AlBa2005}. 

\begin{remark}
A coalescent Markov chain has only one irreducible component $C$ and it is aperiodic. Otherwise, we can find two different components for the double chain $(X_t,Y_t)$ by starting at $(x,x)$ and $(x,y)$ for $x$ and $y$ in different components { (or different positions in the period)} for the original chain. This implies that regardless of the starting value $X_0$, $X_t$ tends in law to the unique stationary distribution with
support on the irreducible component $C$. Note, however, that it is easy to construct Markov chains with a unique irreducible component but that are not coalescent.
\end{remark}

\section{Kesten's tree}

It is helpful to recall convergence of $T_n$ under
a finite variance condition to Kesten's infinite tree $T_\infty$ \cite{Kes86} (see also \cite{LP16}).
Let us first recall the definition of $T_\infty$.
In every generation, starting with the 0-th generation that contains
the root, one node is marked. These marked nodes form an
infinite path called the spine. The number of children of
the node $v_i$ on the spine in generation $i$ is denoted by $\zeta_i$.
The sequence $(\zeta_0, \zeta_1, \ldots)$ is i.i.d.\ with
common distribution $\zeta$ having the size-biased law:
$$
\Pr{ \zeta = i } = ip_i = i\Pr{\xi = i}, \quad i \ge 1.
$$
Observe that $\Ex{ \zeta } = 1 + \sigma^2$.
Furthermore, of the $\zeta_i$ children of $v_i$,
we select a uniform random node to mark as $v_{i+1}$.
The unmarked children of $v_i$ are all roots of
independent unconditional Galton--Watson trees
distributed as $T$.

Convergence of $T_n$ to $T_\infty$ takes place in the following
sense. Let $(T_n , k)$ denote the truncation of $T_n$
to generations $0,1,\ldots,k$. Let $t_k$ denote an arbitrary
finite ordered tree whose last generation is at most $k$.
Then for all $k$ and $t_k$,
$$
\lim_{n\to\infty} \Pr{ (T_n, k) = t_k } =  \Pr{ (T_\infty, k) = t_k }.
$$
The total variation distance between $(T_n,k)$ and $(T_\infty, k)$
is given by
$$
{1 \over 2} \sum_{t_k} \left| \Pr{ (T_n, k) = t_k } - \Pr{ (T_\infty, k) = t_k } \right|.
$$
It is easy to see that this tends to zero as well.

Let us first analyze the root value of $T_\infty$.
It is not at all clear that it is even properly defined
since $T_\infty$ has an infinite path.  However,
the root value is with probability one properly defined under a Markovian condition.
To set this up, we consider a Markov chain on $S$ that runs from
$-\infty$ up the spine to time $0$ (the root), where ``time''
refers to minus the generation number in the Galton--Watson
tree. Let us call $X_{-t}$ the value of node $v_t$ on the
spine. Furthermore, we have
$$
 X_{-t} = f( X_{-t-1}, ~\overline{\!U}_{-t} ),
$$
where $ ~\overline{\!U}_{-t}$ gathers all random { variables} necessary to
compute the value of $v_t$ from that of $v_{t+1}$, i.e.,
$\zeta_t$ (the number of children), $M$ (the index of the marked child), 
the random element $U$, and $W_1, W_2, \ldots$
(the values of the non-marked children, which are
i.i.d.\ and distributed as the value of the root
of an unconditional Galton--Watson tree $T$). 
This is called the {\it spine's Markov chain}.
The Markov chain of Definition~\ref{def:coal},
$$
 (X_{-t}, Y_{-t}) = ( f( X_{-t-1}, ~\overline{\!U}_{-t} ), f( Y_{-t-1}, ~\overline{\!U}_{-t} ) )
$$
is called the {\it spine's double Markov chain}.

\begin{theorem}[\sc limit for kesten's tree.]\label{thm:kestree}
Assume that the spine's Markov chain is coalescent.
Then, the value of the root of $T_\infty$ is with probability one
properly defined.  Furthermore, it is exactly distributed as the
stationary distribution of the spine's Markov chain.
In addition, all values on the spine have the same distribution.
\end{theorem}

\begin{proof}
The proof follows immediately from the coalescent condition along the lines of the proof of Propp and Wilson's theorem { \cite{PW96}} on coupling from the past for explicit Markov chains. 
{ See also~\cite{AlBa2005}}, who have a genuine tree version; here it suffices to follow the infinite spine, so the classical Markov chain setting suffices.
\end{proof}

\medskip

We use the notation $W_\infty$ for a random variable
that is distributed as the stationary distribution of the spine's
random chain.

\medskip

\section{Simulating the root value in tree-based Markov chains.}

Theorem \ref{thm:kestree} has an important algorithmic by-product.
Assume that we wish to generate on a computer a random variable
that is distributed as $W_\infty$.
As a first step, we can write a simple procedure that
generates an unconditional Galton--Watson tree $T$,
associates with all nodes the random elements, 
and computes the root value, $W$. The time taken by
this method is proportional to $|T|$, which is almost surely
finite. In some cases, one can generate $W$ more efficiently
if one knows the distribution on $S$ that solves
the distributional identity
$$
W \inlaw f(W,U) \isdef f_\xi (W_1,\ldots,W_\xi,U),
$$
where $W_1,\ldots,W_\xi$ are i.i.d.\ and distributed as $W$,
and $U$ is the random element.
To simulate the root value of Kesten's tree under the condition of Theorem 1,
we proceed by generating $T_\infty$ iteratively along the spine.
As we process $v_i$, the node on the spine's level $i$,
we generate its random element ($U_i$), its number of children ($\zeta_i$),
its marked child's index ($M_i$, uniformly distributed between $1$ and $\zeta_i$), and
the values $W_{j,i}$ for $1 \le j \le \zeta_i, j \not= M_i$ (which
are i.i.d.\ and distributed as $W$).  As we also have these values
for all the ancestors of $v_i$, we can check the root's
value given that the marked node takes all possible values
in $S=\{1,\ldots,k\}$. If the root's value is unique, then coalescence
has taken place, and thus, the root's value is precisely
distributed as $W_\infty$.
Note that all the random elements generated for each node
stay with the node forever.
Because our Markov chain is coalescent, this procedure
halts with probability one.
This is, in fact, a tree-based version of coupling from the past
{ \cite{PW96,Fil98}.} 

\medskip

\section{The main theorem.}

We are now ready for the main theorem.

\begin{theorem}[\sc limit for $T_n$]\label{thm:main}
Assume that the spine's Markov chain is coalescent.
Then, the value of the root of $T_n$ tends in distribution to
$W_\infty$ as $n \to \infty$.
\end{theorem}

\begin{proof}
We show that for given $\epsilon > 0$, the total variation
distance between $W_\infty$ and the value of the root of $T_n$ 
is less than $\epsilon$. 
First, we invoke { the local convergence of conditioned Galton--Watson trees toward Kesten's tree, see \cite{AD14} for instance}: for any fixed $k$, there exists an $n_k$ such that for
all $n \ge n_k$
the total variation distance between $(T_\infty,k)$
and $(T_n,k)$ is less than $\epsilon/2$.
By Doeblin's coupling theorem \cite{Doe37}, we can find
coupled trees $T_n$ and $T_\infty$ for which 
$$
\Pr{ (T_n,k) \not= (T_\infty,k) } \le {\epsilon \over 2}
$$
for such $n$.  Let $A_{n,k}$ be the bad event, $(T_n,k) \not= (T_\infty,k)$.
Furthermore, on the complement $A_{n,k}^c$, we populate all
nodes in $(T_\infty,k)$ with the missing random elements, i.e., the
$U$'s associated with the nodes. 
Nodes in $(T_n,k)$ 
receive the same random elements as their counterparts in $(T_\infty,k)$. 
Those that live at or past level $k$ are given independent elements.

Define $\ell = \lfloor {k^{1/3}} \rfloor$.
Let $H$ be the maximal height of any subtree rooted at any non-marked
child of $v_0,\ldots,v_\ell$. 
Let $\zeta_i$ be the number of children of $v_i$.
Then, { for $T$ an unconditioned Galton--Watson tree},
\begin{align*}
\Pr{ H \ge k-\ell }  
&\le \Pr{ \sum_{i=0}^\ell ( \zeta_i -1 )  \ge \ell^2 } + \ell^2 \, \Pr{ {\rm height} (T) \ge k-\ell } \\
&\le \Pr { \sum_{i=0}^\ell ( \zeta_i -1 )  \ge \ell^2 } + \ell^2 \times {2+o(1) \over \sigma^2 (k-\ell) } 
\end{align*}
where in the last step, we used Kolomogorov's estimate { \cite{Kol38, KNS66}} \footnote{In the case 
that $\sigma^2=\infty$, the second term should be replaced by $o(1/(k-\ell))$
{ (see \cite{KNS66,Se69})}.}
By the { weak} law of large numbers, and since $\ell\sim k^{1/3}$, we see that the limit 
of the upper bound is zero as $k \to \infty$. 

Consider the values of the nodes $v_0, \ldots, v_\ell$ for
both trees, $T_n$ and $T_\infty$, provided that { $A_{n,k}^c$ holds}. 
Call these $W_{n,0},\ldots,W_{n,\ell}$
and $W_0,\ldots,W_\ell$, respectively.
{ We observe that if $H < k-\ell$,
then $W_{n,0} = W_0$
regardless of whether $W_{n,\ell} = W_\ell$ or not, provided that the spine Markov chain, started
at level $\ell$ coalesces before level $0$.} By our condition,
this happens with probability $1-o(1)$ as $k \to \infty$.
Thus, the probability that the root values of $T_n$ and $T_\infty$ are
different is less than
$$
\Pr{ A_{n,k}} + \Pr{ H \ge k-\ell } + 
 \Pr{ A_{n,k}^c, H < k-\ell, W_{n,\ell} \not= W_\ell , W_{n,0} \not= W_0 }.
 $$
We first choose $k$ large enough to make each of the last two terms less
than $\epsilon / 3$. Having fixed $k$, the first term is smaller than
$\epsilon / 3$ for all $n$ large enough.
Since $W_0$ has the sought limit distribution, we see hat the
total variation distance between $W_{n,0}$ and $W_0$ is not more than
$\Pr{ W_{n,0} \not= W_0 } < \epsilon$.
\end{proof}

\medskip

\section{Applications}

\subsection{Negative example 1: The counting function.}

When 
$$
f_\ell (w_1,\ldots,w_\ell, \,\cdot) \equiv 1+\sum_{i=1}^\ell w_i,
$$
then the root value of $T_n$ is $|T_n| = n$. 
The ``$\!\!\!\mod k$'' version of this function can be considered to force
a finite state space: When 
$$
f_\ell (w_1,\ldots,w_\ell, \,\cdot ) \equiv 1+\sum_{i=1}^\ell w_i \!\!\mod k,
$$
then the root
value of $T_n$ is $n \!\!\!\mod k$. The spine's Markov chain
is not coalescent: when it is started with values $(i,j)
\in \{0,1,\ldots,k-1\}^2$, then all its future values
are of the form $(i+\lambda \!\!\!\mod k, j+\lambda \!\!\!\mod k)$,
so that there are indeed at least $k$ irreducible
components in the chain.

\subsection{Negative example 2: The leaf counter function.}

When 
$$
f_\ell (w_1,\ldots,w_\ell, \, \cdot) \equiv \max \left( 1, \sum_{i=1}^\ell w_i \right),
$$
then the root value of $T_n$ counts $L_n$, the number of leaves in the tree. { As before, we consider in the following the ``$\!\!\!\mod k$'' version. }
Here, the spine's double Markov chain is not coalescent because it has at
least $k$ irreducible components, just as in the first example.
Even though Theorem~\ref{thm:main} does not apply, we know from elsewhere (e.g.,
{ applying Aldous result in \cite{Ald91}}) that $L_n/n \to p_0$ in probability. What we are
saying here is that the much more refined result about the
asymptotic limit law of $L_n \!\!\!\mod k$ for fixed $k$ cannot be
obtained from Theorem~\ref{thm:main}. In particular, when $p_0=p_2=1/2$
(a Catalan tree), $T_n$ is not defined unless $n$ is odd.
In that case, $L_n = (n+1)/2$, and thus, $L_n \!\!\!\mod k = (n+1)/2 \mod k$,
which cycles through the values of $S= \{0,1,\ldots,k-1\}$.

\subsection{Example 3: Length of a random path.}

A random path in a tree is defined by starting 
at the root and going to a random child
until a leaf is reached. The (edge) length of a random path in $T_n$
is called $L_n$. One can once again consider all computations $\!\!\!\mod k$, 
for some arbitrary natural number $k\ge 2$,
but we do not write this explicitly. The recursive function
can be viewed as follows:
$$
f_\ell (w_1,\ldots,w_\ell, u) = 
\begin{cases}
   1 + w_{1+ \lfloor u\ell \rfloor} & \text{if } \ell > 0, \\
   0 & \text{if } \ell=0. \\
\end{cases}
$$
Here $u$ is a uniform $[0,1]$ random variable.
If $f(\cdot, u)$ is the Kesten tree version of this, then
there is coalescence in one step in the Markov chain
if the number of children (recall that it is denoted by $\zeta$ on the spine) 
is more than one, and $\lfloor u\ell \rfloor$ (the child chosen for the random path)
is not equal to the marked node.
The probability of this is 
$$
\Ex{ \left(1 - 1/\zeta \right) } = 1 - \sum_{i=1}^\infty p_i = p_0.
$$
The probability of no coalescence in $t$ steps is
smaller than
\begin{equation}\label{eq:length_coal}
(1-p_0)^t,
\end{equation}
and thus tends to zero.
Thus, Theorem~\ref{thm:main} applies to the length of a random path $\mod k$.
Since the expected length of a random path in an unconditional
Galton--Watson tree is $1/p_0$ and in a Kesten tree is $2/p_0$,
we see that the $\!\!\!\mod k$ can safely be omitted\footnote{What we mean here is that, 
since the sequence $(L_n)_{n\ge 1}$ is tight, the convergence of $\Pr{L_n \!\!\!\mod k =i}$, 
for arbitrary $k$ imply the convergence of $\Pr{L_n = i}$.}. 
The length of
a random path in $T_n$ tends in distribution to the root value
of the Kesten tree.

It is easy to see that for an unconditional Galton--Watson tree $T$,
the random path length ($W$) is geometric with parameter $p_0$, i.e.,
$$
\Pr{ W = i } = p_0 (1-p_0)^i, i \ge 0.
$$
Also, in Kesten's tree, the number of edges traversed 
on the spine is geometric with parameter
$$
\Ex{ \left(1 - 1/\zeta \right) } = 1 - \sum_{i=1}^\infty p_i = p_0.
$$
Thus, $L_n \tendsinlaw W+W'$,
where $W, W'$ are independent geometric$(p_0)$ random variables.

\noindent{\textbf{Remark:} One may replace the use of ``$f \mod k$'' for an arbitrary natural number $k$ by $\min(f,k)$: doing this might simplify the arguments related to tightness, since the convergence of $\min(L_n,k)$ to something of the form $\min(L,k)$ for arbitrary $k$ implies tightness. However, here, our main objective is merely to illustrate the variety of uses of our result, and we rely on the well-known that $(L_n)_{n\ge 1}$ is tight. }

\medskip

\subsection{Example 4: Existence of a transversal in a pruned tree.}

{ Given a tree mark independently and with
probability $p$ each node in the tree.  One may think of a marked node as
a defective node. A transversal of a tree is a collection of nodes
which intersects every path from root to leaf. A transversal is called marked if 
all nodes in it were marked. The main question is that of the existence of a marked transversal in the tree; }this has been used as a model
of breaking up terrorist cells (see Chvatal et al, 2013).

{ It fits in our framework using the following correspondance: }
A marked node has the value one. An unmarked node has
value one if its subtree contains a marked transversal, i.e., if all
the subtrees corresponding to its children contain marked transversals.
The basic recursion for a node with child values $w_1, \ldots, w_\ell$
and uniform element $U$ (which is used for marking) is
$$
w = f_\ell (w_1, \ldots, w_\ell, U)
= 
\begin{cases}
   1  & \text{if $U<p$}, \\
   \prod_{i=1}^\ell w_i & \text{if $U>p$, $\ell>0$}, \\
   0  & \text{if $U>p$ and $\ell=0$}. \\
\end{cases}
$$
The question of existence of a marked transversal then boils down to whether the value at the root is one.
If coalescence does not occur in one step,
then we must have $U > p$. Therefore, the probability of
no coalescence in $t$ steps is not more than $(1-p)^t$,
and we have indeed a coalescent Markov chain to which 
Theorem~\ref{thm:main} applies. When the limit law of $W_\infty$ is worked
out, i.e., { $\rho^*=\Pr{ W_\infty = 1 }$ is identified},
one rediscovers the result of { Devroye in \cite{Dev11}: 
\[\rho^*=\frac{p}{1-(1-p)g'(r)}\qquad \text{with}\qquad  r = p + (1 - p)(g(r)- g(0)),\]
where the fixed-point equation on the right defines $r$ uniquely and $g(s)=\E[s^\xi]$.}

\medskip

\subsection{Example 5: The random child function.}

We define $f_0 (U) = U$, thereby attaching an independent random variable, $U$
to each leaf.
For internal nodes with $\ell$ children, we let $V$ be a uniform $[0,1]$ random variable and
have the recursion
$$
w = f_\ell (w_1,\ldots,w_\ell, V) = w_{1+\lfloor \ell V \rfloor},
$$
the value is that of a uniformly at random chosen child. 
This map percolates one of the leaf values up to the root.
In the spine's Markov chain, coalescence occurs in one step
if, as in the random path length example, a node
does not select its sole marked child. Thus, as in that example, we see from \eqref{eq:length_coal} that 
the probability of not having coalesced in $t$ steps
is not more than $(1-p_0)^t$, and thus, Theorem~\ref{thm:main} also
applies to this case. It should be obvious that $W_\infty \inlaw U$.
(In this case, $W$ is not discrete; the results of Theorem~\ref{thm:main} still apply because 
the coalescence actually does not depend on the actual values at the leaves.)

\medskip

\subsection{Example 6: The minimax function.}

This example follows a model studied by Broutin and Mailler { in \cite{BM17}.}
For each node, we flip a Bernoulli$(p)$ coin to
determine whether the node is a max-node (with probability $p$) or a min-node
(with probability $1-p$).
Max nodes take the maximum of the child values, and min nodes take
the minimum.  In addition, leaf nodes are given a Bernoulli$(q)$
value.
For an unconditional critical Galton--Watson tree, Avis and Devroye
{ (in unpublished work)} showed that the root value is Bernoulli$(p^*)$ where
$p^*$ is the unique solution of the equation
$$
p^* = pp_0 + q (1 - g(1-p^*)) + (1-q) (g(p^*)-p_0),
$$
where we recall that $g(s) = \Ex{ s^\xi }$.


When $p$ and $q$ are both in $(0,1)$, then $p^* \in (0,1)$.
For a max (min) node with $\zeta$ children, we have coalescence in one step if $\zeta > 1$
and the leftmost non-marked child of the node has the value one (zero).
So, the probability of avoiding coalescence in $t$ steps
is not more than
$$
\left( 1 - (1 - p_1) ( pp^*  + (1-p)(1-p^*) )  \right)^t,
$$
and hence we have a coalescent Markov chain when $p, q \in (0,1)$
and $p_1 \not= 1$ { (we precisely excluded the special case $p_1=1$ in the introduction)}.
Note that this result does not require a finite variance for $\xi$.

If $T_n$ is a critical Galton--Watson tree with $p_1 < 1$,
conditioned to be of size $n$,
and if the variance of $\xi$ is finite,
then Theorem~\ref{thm:main} applies. One can compute the limit law of the Markov chain.
In particular, the root value is Bernoulli$(p_n^*)$ where
$$
\lim_{n \to \infty} p_n^* = { q ( 1 - g' (1-p^*) ) \over  1 - q g' (1-p^*) - (1-q) g' (p^*) }.
$$

\medskip

\subsection{Example 7: Random Boolean functions.} 
This is a ``functional version'' of the previous example, which also shows that Theorem~\ref{thm:main} 
also applies to objects that are richer than merely integers. 

Assume for simplicity that $\xi$ is $0$ or $2$ with equal probability, so that $T_n$ is binary. For each node, one flips { an} independent Bernoulli$(p)$ coin to determine { whether it is an} $\AND$-node (with probability $p$) 
or an $\OR$-node (with probability $1-p$). Additionally, the leaves receive one of the { $2k$ 
Boolean literals $x_1, x_2,\dots, x_k, \bar x_1, \bar x_2, \dots, \bar x_k$} independently and uniformly at random (here, $\bar x$ means ``not $x$''). Here, rather than looking at real or Boolean values, we let $S$ be the set of Boolean functions on the variables $x_1,x_2,\dots, x_k$ (so the value of each node is a Boolean function). Then, the value of an $\AND$-node is the Boolean $\AND$ of the values of its children, while an $\OR$ node takes the Boolean $\OR$ of the values of its children. The value at the root is the random Boolean function of $x_1,x_2,\dots, x_k$ that is computed by this ``$\AND$/$\OR$ tree''. 

Note first that, since 
$\AND$/$\OR$ is a complete set of Boolean connectives, every Boolean function 
of $x_1,x_2,\dots, x_k$ can be computed by some finite binary tree with leaves labelled by the corresponding literals $x_1, \bar x_1, \dots, x_k, \bar x_k$. To see that 
the spine's Markov chain is coalescent, observe that the chain coalesces in 
one step if the spine node is an $\AND$ node, and the Boolean 
function computed by the finite tree is identically ``false''; then the node's value is false regardless of the value of its child on the spine. Now the finite tree indeed computes ``false'' with positive probability: one just needs a tree consisting only of an internal node labelled by $\AND$ and leaves, two of which are labelled by $x_i$ and $\bar x_i$, for some $1\le i\le k$. As a consequence, coalescence does not happen in $t$ steps 
with probability 
exponentially small in $t$. It follows that Theorem~\ref{thm:main} applies, 
which proves that the random Boolean function computed at the 
root converges in distribution. Note further that, since the Markov 
chain is irreducible, every Boolean function occurs with positive probability. 
It thus completes results by Broutin and Mailler in \cite{BM17}.

\medskip

\subsection{Example 8: Random binary subtree.} 

One chooses a random binary subtree 
of $T_n$, which contains the root as follows. If the root { has} two children or less, we keep
all of them; otherwise, it has at least three children and we select two uniformly at random
without replacement. One then continues in this fashion at the selected nodes, therefore 
constructing a subtree $T^\star$ of $T_n$ whose nodes all have at most two children. 
If $\xi$ has support contained in $\{0,1,2\}$, the tree $T^\star$ constructed is just
$T_n$, so we suppose that $\Pr{\xi>2}>0$. 
Then, the size (number of nodes) of the subtree $T^\star$ converges in distribution. 

This fits in our framework. Consider first the ``$\!\!\!\mod k$'' version by 
setting $f_0(U)=1$, $f_1(w_1,U)=w_1$, $f_2(w_1,w_2,U)=w_1+w_2 \!\!\! \mod k$ 
 and, for $\ell\ge 3$, 
\[f_\ell(w_1,w_2,\dots, w_\ell,u) = w_{\sigma(u)} + w_{\tau(u)} \!\!\! \mod k,\]
where $(\sigma(u),\tau(u))=(i,j)$ if $u \in A_{i,j}$ for some partition
 $(A_{i,j})_{1\le i<j\le \ell}$ of $[0,1]$ into intervals of equal length.
Observe that, if $\Pr{\xi = 1} = 0$, then the size of $T^\star$ is odd with probability 
one; otherwise it may take any integer value at least three. 

The spine's Markov chain is coalescent: it coalesces in one step if a node does 
not select its unique child lying on the spine; this happens with probability $p>0$, so that 
there is coalescence after $t$ steps with probability at least $1-(1-p)^t$. 
This implies in 
particular that $T^\star$ is actually almost surely finite, so that (see Example 3) there is 
convergence in distribution of the size without the need for the $\!\!\!\mod k$. 

\medskip

\subsection{Example 9: The majority function.}

We consider the much studied majority function model { (see \cite{ODo14}, Chapter 5)}.
We associate with the leaves Bernoulli$(p)$ random variables.
For fixed $k > 0$, we consider a tree in which all nodes have
either $0$ or $2k+1$ children.  By criticality of the Galton--Watson
tree we are studying, this forces $p_{2k+1}=1/(2k+1)$,
$p_0 = 1 - 1/(2k+1)$ and $p_i=0$ for $i \not\in \{ 0, 2k+1 \}$.
At each internal node with $2k+1$ children, we 
take a majority vote among the children.  In other words,
if $x_1,\ldots,x_{2k+1}\in \{0,1\}$ are the binary child values,
then the value at the node is
$$
 \mathbf 1_{\{  2 (x_1+\dots+x_{2k+1}) \ge (2k+1)\}}.
$$
Let us first consider the value $W$ of the root of
an unconditional Galton--Watson tree. If $W$ is Bernoulli$(p^*)$,
then a simple recursion shows that $p^*$ is the solution of
the following recursive equation:
$$
p^* = {1 \over 2k+1} \Pr { 2 \, {\rm Binomial} (2k+1, p^*) > 2k+1 }
    +   {2k \over 2k+1} p.
$$
This yields an equation of degree $2k+1$. The solution $p^*$
increases monotonically from $0$ (at $p=0$) to $1/2$ (at $p=1/2$)
and $1$ (at $p=1$).

Let $W_n$ be the value of the root of $T_n$,
a conditional Galton--Watson tree of size $n$. For $p \in \{ 0,1\}$,
we have $W_n  \in \{ 0,1\}$ accordingly.
So, we assume $p \in (0,1)$.
For an internal node with $2k+1$ children, we have
coalescence in one step if the $2k$ non-marked children are all one.
The probability of this is at least $(p^*)^{2k} > 0$.
So, the probability of avoiding coalescence in $t$ steps
is not more than
$$
\left( 1 - (p^*)^{2k}   \right)^t,
$$
and hence we have a coalescent Markov chain when $p \in (0,1)$.
By Theorem~\ref{thm:main}, $W_n$ tends to a limit random variable.   
In fact, along the spine, we have a simple Markov chain on 
$\{0,1\}$ with transition probabilities $p(0,1)$ and $p(1,0)$
explicitly computable:
\begin{align*}
p(0,1) &= \Pr {{\rm Binomial} (2k, p^*) > k}, \\
p(1,0) &= \Pr {{\rm Binomial} (2k, p^*) < k }. 
\end{align*}
Thus, by well-known results on Markov chains,
$$
\lim_{n\to\infty} \Pr { W_n = 1 } = { p(0,1) \over p(0,1) + p(1,0) }
  = { \Pr { {\rm Binomial} (2k, p^*) > k } \over 1 - \Pr { {\rm Binomial} (2k, p^*) = k } }.
$$

\begin{figure}
	\includegraphics[width=7cm]{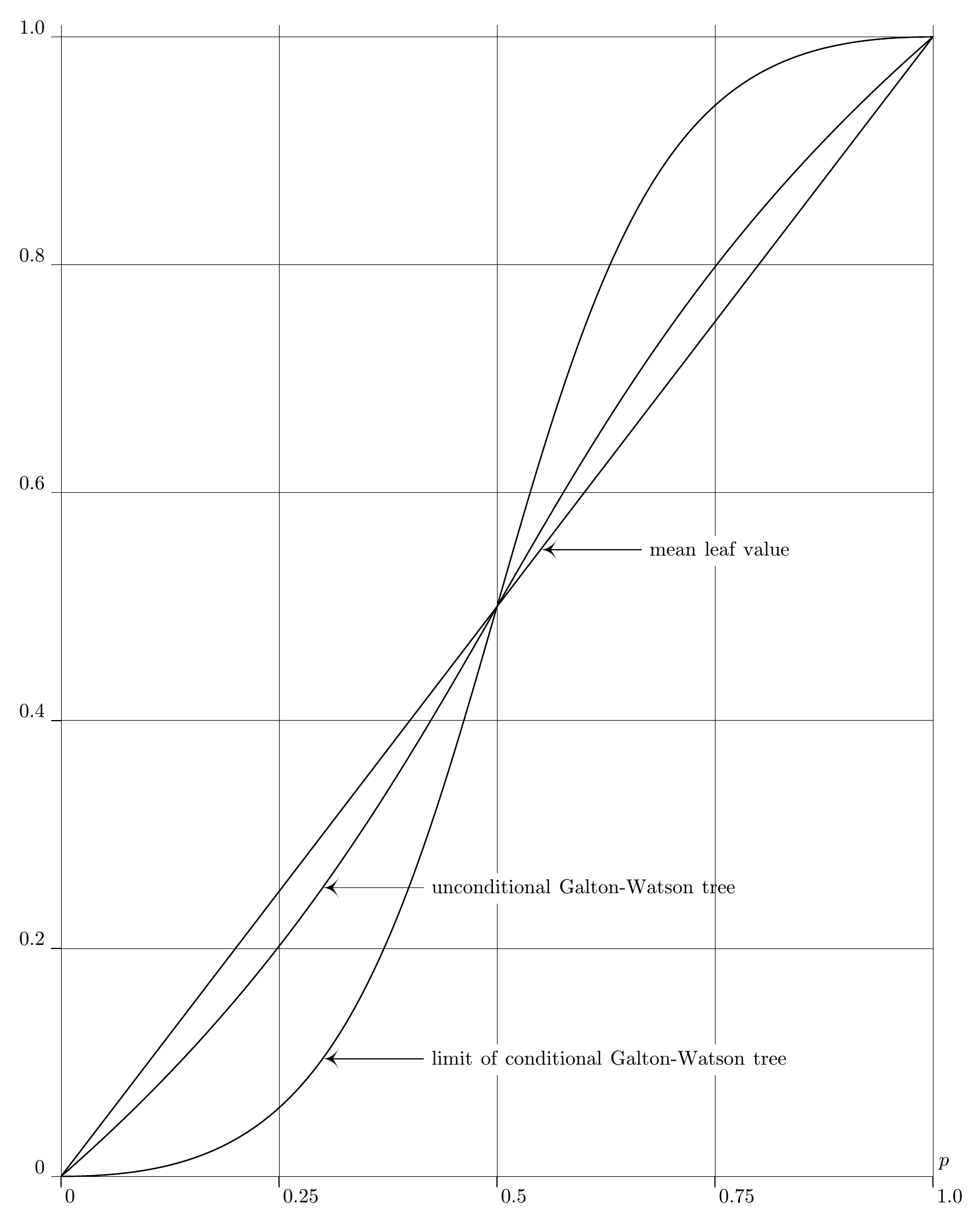}
	\caption{\label{fig:marjority}Consider the case of a ternary
Galton--Watson tree of Example 9 (case $k=1$ in the text).
We show the Bernoulli parameters for the leaf values (the diagonal),
the unconditional Galton-Watson tree,
and the limit of the conditional Galton-Watson tree.}
\end{figure}

\medskip

\subsection{Example 10: The median function.}

Assume that $\xi$ is with probability one either $0$ or odd, so
$\zeta$ is odd. The leaves receive uniform values in a finite set $S$.
Internal nodes take the median of the values of their children.
It is a good exercise to show that the spine's Markov chain is coalescent,
and that Theorem~\ref{thm:main} applies.

\section{Remarks and open questions}

\noindent \emph{i)} We have assumed that the progeny distribution $\xi$ has finite variance for 
the sake of convenience. The local convergence of $T_n$ towards Kesten's tree actually also 
holds in the case when $\Var(\xi)=\infty$ (provided that $\E[\xi]=1$), see for instance, 
Theorem 7.1 of Janson (2012). In this situation, one still has that the size of an unconditioned 
tree $T$ satisfies $|T|<\infty$ almost surely, and the proofs can be extended to this case. 

\noindent \emph{ii)} We have stated our results for conditioned Galton--Watson trees for the sake 
of simplicity. One should easily be convinced that the results remain true under the weaker 
condition that $T_n$ converges locally to an infinite tree such that (1) there is a unique infinite path, 
and (2) the trees hanging from the spine are independent and identically distributed.

\noindent \emph{iii)} It would be interesting to investigate the more general setting 
where the set $S$ may be countably infinite, or an interval of $\R$. 
It seems believable, if $S$ is only countably infinite the result might remain true under an additional 
condition on the positive recurrence of the spine Markov chain. 
The continuous { state} space offers more possibilities for odd behaviors.


\section*{Acknowledgement}

We would like to thank warmly Jan Lukas Igelbrink who found a mistake in a previous version of the paper. 

\providecommand{\bysame}{\leavevmode\hbox to3em{\hrulefill}\thinspace}
\providecommand{\MR}{\relax\ifhmode\unskip\space\fi MR }
\providecommand{\MRhref}[2]{%
  \href{http://www.ams.org/mathscinet-getitem?mr=#1}{#2}
}
\providecommand{\href}[2]{#2}

\end{document}